\documentclass[letterpaper,12pt]{amsart}
\usepackage{amsmath,amssymb,amsxtra,amsthm, amstext, amscd,amsfonts,fancyhdr, hyperref}
\newcommand{\bA}{{\mathbb{A}}}

\newcommand{\bC}{{\mathbb{C}}}

\newcommand{\bF}{{\mathbb{F}}}

\newcommand{\bN}{{\mathbb{N}}}

\newcommand{\bQ}{{\mathbb{Q}}}
\newcommand{\bR}{{\mathbb{R}}}

\newcommand{\bZ}{{\mathbb{Z}}}

\newcommand{\Bm}{{\mathbf{m}}}

\newcommand{\Bx}{{\mathbf{x}}}
\newcommand{\By}{{\mathbf{y}}}

  \newcommand{\B}{{\mathcal{B}}}
  
  \newcommand{\D}{{\mathcal{D}}}
  
  \newcommand{\F}{{\mathcal{F}}}
  \newcommand{\G}{{\mathcal{G}}}
\newcommand{\HH}{\mathcal{H}}

  \newcommand{\N}{{\mathcal{N}}}

  \newcommand{\Q}{{\mathcal{Q}}}
  \newcommand{\R}{{\mathcal{R}}}

\newcommand{\Scal}{{\mathcal{S}}}

\newcommand{\ep}{\varepsilon}

\newcommand{\ol}{\overline}

\newcommand{\upchi}{{\raise.35ex\hbox{$\chi$}}}

\newcommand{\Gal}{\operatorname{Gal}}

\newcommand{\fd}{\mathfrak{d}}

\makeatletter
\@namedef{subjclassname@2010}{%
  \textup{2010} Mathematics Subject Classification}
\makeatother

\newtheorem{theorem}{Theorem}[section]

\newtheorem{lemma}[theorem]{Lemma}

\theoremstyle{definition}

\numberwithin{equation}{section}

\begin{document}

\title{Density of power-free values of polynomials II}

\author{Kostadinka Lapkova}
\address{Institute of Analysis and Number Theory\\
Graz University of Technology \\
Kopernikusgasse 24/II\\
8010 Graz\\
Austria}
\email{lapkova@math.tugraz.at}

\author{Stanley Yao Xiao}
\address{Department of Mathematics \\
University of Toronto \\
Bahen Centre \\
40 St. George St., Room 6290 \\
Toronto, Ontario, Canada M5S 2E4}
\email{stanley.xiao@math.toronto.edu}
\indent


\begin{abstract} In this paper we prove that polynomials $F(x_1, \cdots, x_n) \in \bZ[x_1, \cdots, x_n]$ of degree $d \geq 3$, satisfying certain hypotheses, take on the expected density of $(d-1)$-free values. This extends the authors' earlier result in \cite{LX} where a different method implied the similar statement for polynomials of degree $d\geq 5$. 
\end{abstract}

\date\today

\maketitle


\section{Introduction}

Let $F \in \bZ[x_1, \cdots, x_n]$ be a polynomial of degree $d$. We say that $F$ is \emph{$k$-admissible} if for any prime $p$, there exists $\Bx_p \in \bZ^n$ for which $F(\Bx_p)$ is not divisible by $p^k$ and that $F$ is not divisible over $\bC$ by the square of a non-constant polynomial. In this paper we address the problem of finding the density of integer tuples $\Bx \in \bZ^n$ for which $F(\Bx)$ is indivisible by the $(d-1)$-th power of any positive integer greater than one. 
 In particular, we show that the number of such tuples of box height up to $B$ satisfies an asymptotic formula provided that $F$ is $(d-1)$-admissible; this is the statement of Theorem \ref{MT}. When $d \geq 5$ we have previously obtained the statement of Theorem \ref{MT} in \cite{LX}. \\

We briefly recount the history of the problem of finding, for a given $k$-admissible polynomial $F$ in $n$ variables, the density of $n$-tuples for which $F(\Bx)$ is $k$-free. The case of $n = 1$ is by far the most well-known. Using a simple sieving device it is not difficult to obtain the correct density of square-free values of linear polynomials, but already for the case $k = d$ some innovation is needed. This was first accomplished by Estermann \cite{Est}, who dealt with the density of square-free values of the polynomial $f(x) = x^2 + 1$. Estermann's argument was then shown to apply to all  $k$-admissible polynomials in a single variable whenever $k \geq d$ by Ricci \cite{Ric}. \\

The $k = d-1$ case was first addressed by Erd\"{o}s \cite{Erd}, who showed that the number of integers $1 \leq m \leq B$ for which an admissible polynomial in one variable of degree $d\geq 3$ takes on $(d-1)$-free values tends to infinity as $B$ does. However, he was unable to produce an asymptotic formula using his arguments. The expected asymptotic formula was settled by Hooley in \cite{Hoo0} in the special case of $f(x) = x^3 + k, k \in \bZ$ and in general in \cite{Hoo}. \\ 

Various authors have worked on the case of $k$-free values of polynomials of a single variable, see \cite{LX} and \cite{X} for a summary. In \cite{Gran}, A.~Granville demonstrated that the question of square-free values of polynomials is intimately related to the well-known \emph{abc}-conjecture. He showed that the analogous asymptotic relation given in (\ref{MT asy}) holds for square-free values for all admissible polynomials, provided that the \emph{abc}-conjecture holds. Poonen in \cite{Poo} then vastly generalized Granville's result to show that the same holds for admissible polynomials in arbitrarily many variables. However, Poonen's argument, similarly to the method of Erd\"{o}s, 
does not allow one to obtain the expected asymptotic formula (\ref{MT asy}). He further showed that, without the \emph{abc}-conjecture, one can still unconditionally obtain $k$-free values of polynomials provided that the degree of the polynomial $F$ is sufficiently small. In other words, if $k,d$ are related in a way that allows one to conclude that admissible degree $d$ polynomials in $\bZ[x]$ take on infinitely many $k$-free values, then the same holds for admissible degree $d$ polynomials in $\bZ[x_1, \cdots, x_n]$ for any $n\geq 2$. \\


This aspect of Poonen's paper is confirmed in our previous paper \cite{LX}, where a relatively straightforward reduction to the single-variable case is possible because of the uniformity produced by the determinant method. For the $k = d-1$ case more care is necessary. Indeed, we shall adopt a sieve argument introduced by Hooley in \cite{Hoo1}. For the middle range we use the sieving technique of Ekedahl in the same way as in our previous paper \cite{LX}.\\



Put
\begin{equation} \label{count} N_{F}(B) = \# \{\Bx \in \bZ^n : |\Bx| \leq B, F(\Bx) \text{ is } (d-1)\text{-free}\}
\end{equation} 
and
\begin{equation} \label{rhoF} \rho_F(m) = \# \{\Bx \in (\bZ/m \bZ)^n : F(\Bx) \equiv 0 \pmod{m}\}.
\end{equation}

We have the following theorem:

\begin{theorem} \label{MT} Let $F(\Bx) \in \bZ[x_1, \cdots, x_n]$ be a square-free polynomial of degree $d \geq 3$ such that for all primes $p$, there exists $\By \in \bZ^n$ satisfying $p^{d-1} \nmid F(\By)$.  
Then for all $B \in \bR_{> 0}$ we have the asymptotic relation
\begin{equation} \label{MT asy} N_F(B) = C_F B^n + o_F\left(B^n \right),\end{equation}
where the positive constant $C_F$ is defined by
\[C_F=\prod_p \left(1 - \frac{\rho_F(p^{d-1})}{p^{n(d-1)}} \right).\]

\end{theorem} 

The error term in the theorem is too weak for us to confirm the analogous relation for prime inputs, as we could do in \cite{LX}. \\ 

The proof of the theorem is divided in three sections. The first one recreates the setting of the simple sieve and the preliminary work similar to the one in \cite{LX}. In section \S\ref{Ekedahl} we formulate the sieve of Ekedahl and apply it once for estimating the middle range of primes and another time in order to reduce our task to dealing with polynomials which are geometrically irreducible. This will simplify the argument in the last section \S\ref{N3} where at a certain point we apply the Lang-Weil bound for geometrically irreducible polynomials. The technique in the last range follows ideas from Hooley's work \cite{Hoo1} in the simpler setting $k=d-1$. In particular, we will avoid use of the Prime Ideal Theorem (e.g. Lemma 3 and Lemma 4 in \cite{Hoo1}) and elementary arguments would suffice. \\

We remark that for the middle range one could also apply the stratification result with upper bounds of exponential sums modulo $p^2$ that we proved in \cite{LX_exp}, if we insert the latter result in Hooley's argument from \cite{Hoo1}. However, at present this only applies to homogeneous polynomials, so in the current paper we apply the Ekedahl sieve in line with our argument in \cite{LX}.      


\section{Preliminaries}
\label{Prelim}

We will show that $N_{F}(B)$ (recall (\ref{count})) satisfies an inequality of the form
\begin{equation} \label{simple sieve} N_1(B) - N_2(B) - N_3(B) \leq N_{F}(B) \leq N_1(B).
\end{equation} 
Our goal will be to demonstrate that for any $\ep > 0$, we have
\[N_1(B) = B^n \prod_{p \leq \xi_1} \left(1 - \frac{\rho_F(p^{d-1})}{p^{n(d-1)}}\right) + O_{F,\ep} \left(B^{n - 1 + \ep}\right),\]
and for some $\delta_n > 0$ and some slowly growing function $\xi_1 = \xi_1(B)$ tending to infinity as the parameter $B$ tends to infinity, one has
\[N_2(B) = O_F\left(B^n \left(\xi_1^{-1} + (\log B)^{-\delta_n} \right) \right)\]
and 
\[N_3(B) = o_F(B^n).\]
Let us define the iterated logarithm function for positive $B$ by $\log_1 (B) = \max\{1, \log B\}$ and $\log_s B = \log_1 \log_{s-1} B$ for $s \geq 2$. We now suppose that $\xi_1(B)$ satisfies
\begin{equation} \label{xi1} \xi_1 = \xi_1(B)=O(\log_3 B/\log_4 B), \end{equation}
and 
\begin{equation} \label{xi2} \xi_2 = \xi_2(B)= B (\log B)^{\delta_n},\end{equation}
where $\delta_n$ depends on $n$ only and $\delta_n > 0$ for all $n$. \\

We then put
\begin{equation} \label{M1} N_1(B) = \# \{\Bx \in \bZ^n : |\Bx| \leq B,\quad p^{d-1} | F(\Bx)\implies p > \xi_1 \},\end{equation}
\begin{equation} \label{M2} N_2(B) = \# \{\Bx \in \bZ^n : |\Bx| \leq B, p^{d-1} | F(\Bx)\implies p > \xi_1,\end{equation}
\[\exists \xi_1<p \leq \xi_2 \text{ s.t. } p^2 | F(\Bx)\},  \]

and
\begin{equation} \label{M3} N_3(B) = \# \{\Bx \in \bZ^n : |\Bx| \leq B, p^{d-1} | F(\Bx)\implies p > \xi_1,\end{equation}
\[p^2 \nmid F(\Bx) \text{ for } \xi_1 < p \leq \xi_2 \text{ and } \exists p > \xi_2 \text{ s.t. } p^{d-1} | F(\Bx)\}\}.\]

Before we proceed with estimating $N_1(B)$, let us state some facts about the function $\rho_F$ as defined in (\ref{rhoF}). We need the following lemma: 
\begin{lemma} \label{rhoF est} Let $F$ be a square-free polynomial in $n$ variables with integer coefficients, and such that for all primes $p$, $p^k$ does not divide $F$ identically. Then for any $k\geq 2$ we have $\rho_F(p^k) = O_F\left(p^{nk-2}\right)$. 
\end{lemma}

\begin{proof}
This is Lemma 2.1 from \cite{LX}.
\end{proof}

We remark that Lemma \ref{rhoF est} implies that the infinite product
\[\prod_p \left(1 - \frac{\rho_F(p^{d-1})}{p^{n(d-1)}} \right)\]
converges. This is because 
\[\frac{\rho_F(p^{d-1})}{p^{n(d-1)}} = O \left(\frac{1}{p^2}\right),\]
by Lemma \ref{rhoF est}. \\

We give an estimate for $N_1(B)$. Define, for a positive integer $b$, the quantity
\[N(b,B) = \# \{\Bx \in \bZ^n \cap [-B,B]^n : b^{d-1} | F(\Bx)\}.\]
Then from the familiar property of the M\"obius function $\mu$, we have
\begin{align*} N_1(B) & = \sum_{\substack{ b \in \bN \\ p | b \Rightarrow p \leq \xi_1}} \mu(b) N(b,B) \\
& = \sum_{\substack{ b \in \bN \\ p | b \Rightarrow p \leq \xi_1}} \mu(b) \rho_F(b^{d-1}) \left(\frac{B^n}{b^{n(d-1)}} + O\left(\frac{B^{n-1}}{b^{(n-1)(d-1)}} + 1\right) \right) \\
& = B^n \prod_{p \leq \xi_1} \left(1 - \frac{\rho_F(p^{d-1})}{p^{n(d-1)}}\right) + O\left(\sum_{\substack{ b \in \bN \\ p | b \Rightarrow p \leq \xi_1}} \rho_F(b^{d-1}) \left(\frac{B^{n-1}}{b^{(n-1)(d-1)}} + 1 \right) \right).
\end{align*}
By the theorem of Rosser and Schoenfeld \cite{RS}, it follows that for all $\ep > 0$ and some $C' > 0$ we have
\[\prod_{p \leq \xi_1} p \leq e^{2 \xi_1} = O \left((\log_2 B)^{\frac{C'}{\log_4 B}} \right) = O_\ep (B^\ep),\]
by (\ref{xi1}). Hence, we obtain via Lemma \ref{rhoF est} that, for any $\ep > 0$,
\[N_1(B) = B^n \prod_{p \leq \xi_1} \left(1 - \frac{\rho_F(p^{d-1})}{p^{n(d-1)}}\right) + O\left(\sum_{b \ll_\ep B^\ep} B^{n-1+\ep} + b^{n(d-1) - 2 +\ep}  \right).\]
We then see that
\begin{equation} \label{N1B main estimate} N_1(B) = B^n \prod_{p \leq \xi_1} \left(1 - \frac{\rho_F(p^{d-1})}{p^{n(d-1)}}\right) +  O_{\ep} \left(B^{n-1 + \ep}  \right).\end{equation}
As $B \rightarrow \infty$, the partial product in (\ref{N1B main estimate}) tends to the convergent product in Theorem \ref{MT}, thus it suffices to show that $N_2(B), N_3(B)$ are error terms. 
\section{The Ekedahl sieve and the estimation of $N_2(B)$}
\label{Ekedahl}

In this section, we use a result of Ekedahl \cite{Eke} to handle several situations. As a consequence, we also give an acceptable estimate for $N_2(B)$. The version below is formulated by Bhargava and Shankar in \cite{BhaSha}: 

\begin{lemma}[Ekedahl sieve] \label{Ekedahl sieve} Let $\B$ be a compact region in $\bR^n$ having finite measure, and let $Y$ be any closed subscheme of $\bA_\bZ^n$ of co-dimension $s \geq 2$. Let $r$ and $M$ be positive real numbers. Then we have 
\[\#\{\Bx \in r \B \cap \bZ^n : \Bx \pmod{p} \in Y(\bF_p) \text{ for some prime } p > M\} \]
\[ = O \left(\frac{r^n}{M^{s-1} \log M} + r^{n - s + 1} \right).\]
\end{lemma}

%

Our first application of Lemma \ref{Ekedahl sieve} is to reduce to the case when $F$ is irreducible over $\bQ$. Suppose that $F$ is reducible over $\bQ$, say 
\[F(x_1, \cdots, x_n) = \F_1(\Bx) \cdots \F_r (\Bx),\]
where each $\F_i$ is irreducible over $\bQ$ for $i = 1, \cdots, r$. We put $d^\prime = \max_{1 \leq j \leq r} \deg \F_j$. Let us write 
\begin{equation} \label{N3j B} N_3^{(j)}(B) = \# \{\Bx \in \bZ^n : |x_i| \leq B,\quad p^k | \F_j(\Bx) \implies p > \xi_1,  \end{equation}
\[p^2 \nmid \F_j(\Bx) \text{ for } \xi_1 < p \leq \xi_2, \text{ and }\exists p > \xi_2 \text{ s.t. } p^k | \F_j(\Bx) \}.\]
If $\Bx$ is counted by $N_3(B)$ but not by $N_3^{(j)}(B)$, then $\Bx$ must vanish mod $p$ for two of the $\F_i$'s. Put $V_i$ for the variety defined by the vanishing of $\F_i$; note that $V_i$ is defined over $\bZ$ for $i = 1, \cdots, r$. We claim that the scheme-theoretic intersection $V_{i,j} = V_i \cap V_j$ for $i \ne j$ must be co-dimension $2$ in $\bA^n(\bZ)$. Indeed, if $V_{i,j}$ does not have co-dimension $2$ then $V_i, V_j$ must share a top dimensional component. Therefore $\F_i, \F_j$ must share a geometrically integral factor $G$ defined over $\ol{\bQ}$. However $\F_i, \F_j$ are defined over $\bQ$, so this implies that they must share a factor $\G$ defined over $\bQ$. But then $\G^2 | F$, which contradicts our hypothesis that $F$ is square-free over $\bC$. This confirms the claim that $V_{i,j} $ has co-dimension $2$. \\

Ekedahl's sieve then gives the same bound as in (\ref{Eke bd}) for the number of such points in $[-B,B]^n \cap \bZ^n$. It therefore suffices to deal with the case when $F$ is irreducible over $\bQ$ and $d = \deg F$. \\ 

We now address the situation when $F$ is not geometrically integral. We will show that we can reduce our analysis over $\bF_p$ to points belonging to only one geometrically irreducible component definable over $\bF_p$. This will be useful later when we apply the Lang-Weil bounds. Since we may now assume that $F$ is irreducible over $\bQ$, it follows that over $\ol{\bQ}$, $F$ must admit a factorization as 
\begin{equation} \label{Qbar fact} F(\Bx)=\prod_{j=1}^k F_j(\Bx)\end{equation}
where $F_j$ are absolutely irreducible conjugate polynomials defined over $\ol{\bQ}$ having the same degree $d^\prime$, and $k = d/d^\prime$. For a given prime $p$, let $\tau(p)$ be the number of $F_i$'s defined over $\bQ_p$. Then (\ref{Qbar fact}) turns into
\begin{equation} \label{modp fact} F(\Bx) = F_p^\ast(\Bx) \prod_{j=1}^{\tau(p)} G_j^{(p)}(\Bx)
\end{equation}
where $F_p^\ast \in \bZ_p[x_1, \cdots, x_n]$ and $G_j^{(p)} = F_{k_j}$ for some $k_j$, such that $G_j^{(p)}$ have $\bZ_p$-coefficients. Now, if $\Bx \in \bZ^n$ is counted by $N_2(B), N_3(B)$ then $F(\Bx) \equiv 0 \pmod{p^2}$ for some $p>\xi_1$. The solutions to the congruence $F(\Bx) \equiv 0 \pmod{p^2}$ then must satisfy one of the situations below: 
\begin{itemize}
\item[(a)] There exist $1\leq i<j\leq\tau(p)$ such that $G_i^{(p)}(\Bx)\equiv G_j^{(p)}(\Bx)\equiv 0\pmod p$, 
\item[(b)] There exists uniquely an index $1 \leq i \leq \tau(p)$ such that $G_j^{(p)}(\Bx) \equiv 0 \pmod{p^2}$, or
\item[(c)] $F_p^*(\Bx)\equiv 0\pmod{p} .$
\end{itemize}
We remark that condition (c) is a mod $p$ condition rather than a mod $p^2$ condition. That this suffices will be made clear in the proof of Lemma \ref{lem_ddagger}. \\ 

We put
\begin{equation} \label{Sp} \Scal_p = \{\Bx \in (\bZ/p\bZ)^n : \Bx \text{ satisfies (a) or (c)}\}.\end{equation}
Next put
\begin{equation} \label{N ddagger} N^\ddagger (B) = \# \{\Bx \in \bZ^n \cap [-B,B]^n :  p^k | F(\Bx) \Rightarrow p > \xi_1, \exists p > \xi_1 \text{ s.t. } \Bx \pmod{p} \in \Scal_p \}.\end{equation}
Another application of Ekedahl's sieve then gives the following estimate:
\begin{lemma}\label{lem_ddagger} We have 
\[N^\ddagger(B) = O \left(\frac{B^n}{\xi_1 \log \xi_1} + B^{n-1} \right).\]
\end{lemma} 

\begin{proof}
We show that points considered in (a) and (c) lie on a uniformly bounded number of co-dimension 2 sub-schemes of $\bA^n(\bZ)$ (see also the proof of Lemma 3.1 in \cite{X2}). This claim is clear for (a), since $G_i^{(p)}, G_j^{(p)}$ are equal to $F_{k_i}, F_{k_j}$ in (\ref{Qbar fact}) for some $k_i \ne k_j$, and these are not proportional since we assumed that $F$ is square-free as a polynomial. \\ \\
Now suppose that $\Bx$ is such that $F_p^\ast(\Bx) \equiv 0 \pmod{p}$ for some prime $p > \xi_1$. Then there exists a factor $\G$ of $F_p^\ast$ defined and irreducible over $\bF_p$ 
 such that $\G(\Bx) \equiv 0 \pmod{p}$. Since by assumption $\G$ is irreducible over $\bF_p$ but has non-trivial factors over $\ol{\bF_p}$, the absolute Galois group $\Gal(\ol{\bF_p}/\bF_p)$ acts non-trivially on the factors of $\G$. Since $\Bx \in \bZ^n$ implies that $\Bx \pmod{p} \in \bF_p^n$, it follows that $\Bx \pmod{p}$ remains fixed by this action. Hence $\Bx \pmod{p}$ must satisfy 
\[H_1(\Bx) \equiv H_2 (\Bx) \equiv 0 \pmod{p}\]
for some distinct conjugate factors $H_1, H_2$ of $F_p^\ast$ defined over $\ol{\bF_p}$. Therefore, $\Bx$ again lies on a uniformly bounded collection of co-dimension 2 sub-schemes. \\ \\
The necessary estimate in the lemma then follows from Lemma \ref{Ekedahl sieve}. 
\end{proof}

Our next application of Ekedahl's sieve follows the same argument as in \cite{LX}. Let $G(\Bx) = \frac{\partial F}{\partial x_1} (\Bx)$. Define the variety $V_{F,G}$ to be 
\begin{equation} \label{VFG} V_{F,G} = \{\Bx \in \bC^n : F(\Bx) = G(\Bx) = 0\}.
\end{equation}
Observe that $V_{F,G}$ is of co-dimension two and is defined over $\bZ$. Put 
\[N^\ast(p; B) = \{\Bx \in \bZ^n : \lVert \Bx \rVert \leq B, \Bx \pmod{p} \in V_{F,G}(\bF_p)\}.\]
It follows from Lemma \ref{Ekedahl sieve} that
\begin{equation}\label{Eke bd}N^\ast(B)=\# \bigcup_{p > \xi_1} N^\ast(p; B) = O \left(\frac{B^n}{\xi_1 \log \xi_1} + B^{n-1} \right).\end{equation}

\subsection{Estimating $N_2(B)$} Put
\[N^\dagger(p; B) = \#\{\Bx \in \bZ^n : \lVert \Bx \rVert \leq B, p^2 | F(\Bx), p \nmid G(\Bx)\}.\]
It then follows that
\[N_2(B) \leq \sum_{\xi_1 < p \leq \xi_2} N^\dagger(p; B) + \sum_{\xi_1 < p \leq \xi_2} N^\ast(p; B),\]
and since we have already estimated the second sum by $N^*(B)$, it suffices to estimate the former. \\

Let us now define the function, for a polynomial $f$ in a single variable $x$, by
\[\rho_f(m) = \#\{s \in \bZ/m\bZ : f(s) \equiv 0 \pmod{m}\}.\]
It is clear from the Chinese Remainder Theorem that $\rho_f$ is multiplicative. Observe that $\rho_f(p) \leq d$ if $p$ does not divide all coefficients of $f$ and in this case the bound is independent of $p$. Moreover, if $p \nmid \Delta(f)$, then for any positive integer $k$ we have $\rho_f(p^k) \leq d$ by Hensel's lemma and this is independent of the coefficients of $f$.\\

Now for fixed $(x_2, \cdots, x_n)$, the solutions to $f(x) = F(x, x_2, \cdots, x_n) \equiv 0 \pmod{p^2}$ contributing to $N_p^\dagger(X)$ must satisfy $p \nmid \Delta(f)$; in particular, the number of solutions in $\bZ/p^2\bZ$ is at most $d$.  We then have that
\begin{equation} \label{Eke con} \sum_{\xi_1 < p \leq \xi_2} N_p^\dagger(B) = O_d \left(B^{n-1} \sum_{\xi_1 < p \leq \xi_2} \left(\frac{B}{p^2} + 1 \right) \right) = O_d \left(\frac{B^n}{\xi_1} + \frac{B^{n-1} \xi_2}{\log \xi_2}\right).\end{equation}
We recall that $\xi_2 = B (\log B)^{\delta_n}$ and we can choose $\delta_n=1/n$, and $\xi_1 = \xi_1(B)$ is a function which tends to infinity, so that (\ref{Eke con}) gives an acceptable contribution to $N_2(B)$. 




\section{Estimating $N_3(B)$: completing the proof of Theorem \ref{MT}}
\label{N3}


In this section we show that $N_3(B) = o(B^n)$ and thus finish the proof of Theorem \ref{MT}. For convenience we sometimes write $k=d-1$ in analogy of the problem for $k$-free values treated in \cite{LX}.\\

In view of Lemma \ref{lem_ddagger}, it suffices to consider the set 
\begin{equation} \label{final error} \N_3^\dagger(B) =  \lbrace \Bx \in \bZ^n : |\Bx| \leq B, F(\Bx) = uq^k \text{ for a prime } q>\xi_2 \text{ and } p^{k} | u\implies p > \xi_1, \end{equation}
\[p^2 \nmid F(\Bx) \text{ for } \xi_1 < p \leq \xi_2, \text{ and } \Bx \pmod{p} \notin \Scal_p \quad\forall p\mid u\text{ such that }p>\xi_1\}\]
and put $N_3^\dagger(B) = \# \N_3^\dagger(B)$. Observe that
\begin{equation} \label{N3 decomp} N_3(B) \leq N_3^\dagger(B) + N^\ddagger(B),\end{equation} 
therefore it is enough to estimate $N_3^\dagger(B)$. 
We shall establish the following preliminary result: 

\begin{lemma} \label{u factor} Let $\Bx \in \N_3^\dagger(B)$ and $u,q$ be as in (\ref{final error}). Then we have
	\[u = O\left(B (\log B)^{-(d-1)\delta_n} \right).\]
	Furthermore, for $B$ large enough $u$ can be written as $u = u_1 u_2$, where $u_1$ divides
	\[C(\xi_1) = \prod_{p \leq \xi_1} p^{d-2},\]
	and $u_2$ is square-free with each prime divisor $p$ of $u_2$ satisfying $\xi_1 < p \leq \xi_2$. 
\end{lemma}

\begin{proof}
	Observe that from $F(\Bx) = uq^{d-1}$ and our assumptions on $q$, we have
	\[u = O\left(B^d \xi_2^{-(d-1)} \right).\]
	By (\ref{final error}) and (\ref{xi2}), there exists an absolute positive constant $C_1$ such that
	\begin{align*} |u| & < C_1 B^{d - (d-1)} (\log B)^{-(d-1)\delta_n} \\
	& = C_1 B (\log B)^{-(d-1)\delta_n}.
	\end{align*}
	We now factor $u$ into two factors $u_1$ and $u_2$, where $u_1$ consists of only prime factors less than $\xi_1$. We observe that since we have accounted for small prime powers via our treatment of $N_1(B)$, we have that $u_1$ divides $\prod_{p \leq \xi_1} p^{k-1}.$  The factor $u_2$, then, will be composed of prime factors larger than $\xi_1$. Further, it must be \emph{square-free}. This is because, by definition, the prime factors of $u$ between $\xi_1$ and $\xi_2$ divide $u$ exactly once, and $u$ cannot have a prime factor exceeding $\xi_2$ for otherwise
	\[uq^{d-1} \geq B^{d} \left(\log B\right)^{d\delta_n},\]
	which contradicts $\Bx \in [-B,B]^n$ for $B$ sufficiently large. \end{proof}
For each square-free integer $u_2$ such that each prime divisor $p$ of $u_2$ satisfies $\xi_1 < p \leq \xi_2$, put
\begin{equation} \label{big D}\D(u_2) = \prod_{\substack{\xi_1 < p \leq \frac{1}{12} \log(B u_2^{-1}) \\ p \nmid u_2 \\ p \equiv 1 \pmod{k}}} p \end{equation} 
and put $\D(u_2)=1$ if the product is empty. We then have the following lemma:
\begin{lemma} \label{big D lem} Let $u_2$ be a square-free integer such that all of its prime divisors are between $\xi_1$ and $\xi_2$. Let $\omega(m)$ denote the number of distinct prime divisors of $m$. Let $\D(u_2)$ be as in (\ref{big D}). If $q > \xi_2$ is a prime, then there exists exactly $k^{\omega(\D)}$ residue classes $\{\fd_1, \cdots, \fd_{k^{\omega(D)}} \}$ such that
	\[\fd_j^k \equiv q^k \pmod{\D}\]
	for $j = 1, \cdots, k^{\omega(D)}$. 
\end{lemma} 

\begin{proof} Since all prime divisors of $\D$ are $O(\log B)$, it follows that $q^k$ is a proper $k$-th power residue modulo $\D$. Now consider the family of all $k$-th power residues modulo $\D$. By our choice of $\D$, we have that $k | \varphi(\D)$, so that the family of $k$-th power residues is not the set of all residues modulo $\D$. For each $p | \D$, $q^k$ has $k$ pre-images modulo $p$, meaning there exist $k$ distinct elements $\mathfrak{q}_1, \cdots, \mathfrak{q}_k$ in $\{0, 1, \cdots, p-1\}$ such that $\mathfrak{q}_j^k \equiv q^k \pmod{\Q}$. For a positive integer $l$ let us write $\omega(l)$ for the number of distinct prime divisors of $l$. Then it follows from the Chinese Remainder Theorem that there exist $k^{\omega(\D)}$ residue classes $\{\mathfrak{d}_1, \cdots, \mathfrak{d}_{k^{\omega(\D)}}\}$ modulo $\D$ such that $\mathfrak{d}_j^k \equiv q^k \pmod{\D}$. \end{proof}

Let $C_1$ be as in Lemma \ref{u factor}, and put 
\begin{equation}\label{xi3}\xi_3 = C_1 B (\log B)^{-(d-1)\delta_n}.
\end{equation} 
Lemmas \ref{u factor} and \ref{big D lem} have the following consequence, which is crucial for the estimation of $N_3(B)$:

\begin{lemma} \label{H bound} Let $u_1$ be a divisor of $C(\xi_1)$ and let $u_2\leq \xi_3$ be a square-free integer whose prime divisors $p$ satisfy $\xi_1 < p \leq \xi_2$. Let $H_{u_1, u_2}(B)$ be the number of solutions $(m_1, \cdots, m_n) \in \bZ^n \cap [-B,B]^n$ to the following three congruences:
	\begin{equation} \label{l-1} F(m_1, \cdots, m_n) \equiv 0 \pmod{u_1},
	\end{equation}
	\begin{equation} \label{l-2} F(m_1, \cdots, m_n) \equiv 0 \pmod{u_2},
	\end{equation}
	and for $0 \leq s < \D$, the solutions to the congruences
	\begin{equation} \label{script D} F(m_1, \cdots, m_n) \equiv u_1 u_2 s^{d-1} \pmod{\D}.
	\end{equation}
	such that $(m_1, \cdots, m_n) \pmod{p}  \in (\bZ/p\bZ)^n \setminus \Scal_p$ for $p |  u_2$. 
Then we have 
	\begin{equation} \label{k-cover} N_3^\dagger(B) \leq \sum_{\substack{u_1 | C(\xi_1) \\ u_2 \leq \xi_3}} \frac{ H_{u_1, u_2}(B)}{(d-1)^{\omega(\D)}} .\end{equation}
\end{lemma} 
\begin{proof} (\ref{k-cover}) follows from the fact that the solutions to (\ref{script D}) can be partitioned into sets of cardinality $(d-1)^{\omega(\D)}$ by Lemma \ref{big D lem}.  
\end{proof} 

We can directly estimate the quantity
\[N_4(B) = \sum_{\substack{u_1 | C(\xi_1) \\ u_2 \leq \xi_3}} \frac{H_{u_1,u_2}(B)}{(d-1)^{\omega(\D)}}.\]
For comparison, the analogous quantity in \cite{Hoo1} and \cite{X2} is estimated using the Selberg sieve. However, in the current case the modulus $u_1 u_2 \D$, due to our restriction on $u_2$, or rather, the size of $\xi_3$, is smaller than $B$, and we can make do without the Selberg sieve. \\

We write $\HH(u_1, u_2)$ for the number of solutions modulo $u_1 u_2 \D$ to the congruences (\ref{l-1}), (\ref{l-2}), and (\ref{script D}). Then plainly we have
\[ H_{u_1, u_2}(B) = O \left(\HH(u_1, u_2) \left(\frac{B}{u_1 u_2 \D} + 1 \right)^n \right).
\]
With the choice of our parameters, e.g. \eqref{xi1}, \eqref{xi2} and \eqref{xi3}, we have  $u_1 u_2 \D\ll B$. Indeed, $u_1\leq C(\xi_1)\ll \log_2(B)^{O((d-2)/\log_4(B))}, \D\ll e^{2\frac 1{12}\log(B/u_2)}\ll (B/u_2)^{1/6}$. We have $u_2\leq \xi_3=C_1 B(\log B)^{-(d-1)\delta_n}$, so $u_2\D\ll u_2^{5/6}B^{1/6}\ll B(\log B)^{-\frac 5 6 (d-1)\delta_n}$ and therefore $u_1 u_2 \D\ll B$. Thus we can ignore the constant term and conclude that
\begin{equation} H_{u_1, u_2}(B) = O \left(\HH(u_1, u_2) \left(\frac{B}{u_1 u_2 \D}  \right)^n \right).
\end{equation}
Note that in \cite{Hoo1} and \cite{X2} the modulus $u_1 u_2 \D\gg B$ and the constant term is not absorbed by the fraction $B/(u_1 u_2 \D)$. In our case since $u_1 u_2 \D\ll B$, the solutions counted by $\HH(u_1,u_2)$ are bounded trivially by $B^n$ while in \cite{X2}, say, they are close to $B^{2n}$, and  the latter requires more subtle sieving with Selberg weights and exponential sums estimates. In our case a direct application of Lang-Weil bound suffices to get the necessary saving from the main term \eqref{N1B main estimate}.\\


Let us now evaluate the quantity $\HH(u_1,u_2)$ by first estimating trivially the number of solutions of \eqref{l-1}. We use the Lang-Weil theorem for estimating the number of solutions of (\ref{l-2}), and (\ref{script D}). Note that in both cases we have geometrically irreducible varieties, in the case of (\ref{l-2}) it follows from the condition $\Bm\pmod p\notin\Scal_p$ for $p\mid u_2$. Something more, if $F$ defines a geometrically irreducible hypersurface over $\bF_p^n$ then Lang-Weil yields 
\[\rho_F(p)=p^{n-1}+ O_F(p^{n-3/2}).\]
From the multiplicativity of the function $\rho_F$ for square-free $m$ we then have
\begin{align*}\rho_F(m)&=\prod_{p|m}\left(p^{n-1}+O_F(p^{n-3/2})\right)\\
&=m^{n-1}\prod_{p|m}\left(1+O_F(p^{-1/2})\right)\\
&=O_F(m^{n-1}\sigma_{-1/4}(m)),
\end{align*}
where we write
\[\sigma_{\alpha}(m) = \sum_{s | m} s^{\alpha}\]
for a positive integer $m$ and a real number $\alpha$.\\

Then for the number of solutions of \eqref{l-2} we have $\rho_F(u_2)=O_F(u_2^{n-1}\sigma_{-1/4}(u_2))$ and for the number of solutions of \eqref{script D} we have $\rho_{F,s}(\D)=O_F(\D^{n-1}\sigma_{-1/4}(\D))$ for each integer $s\in\left [0,\D \right)$. Also note that the dependence on $F$ in the Lang-Weil theorem is on geometric characteristics of the hypersurface considered and uniform in $s$. Thus we obtain
\[\HH(u_1, u_2) = O \left(u_1^n u_2^{n-1} \sigma_{-1/4}(u_2) \D^n \sigma_{-1/4}(\D) \right).\]

Therefore we have 
\begin{align*} N_4(B) & = O\left(\sum_{\substack{u_1 | C(\xi_1) \\ u_2 \leq \xi_3}} \frac{B^n   u_2^{n-1}  \sigma_{-1/4}(u_2) \D^n  \sigma_{-1/4}(\D)}{(u_2 \D)^n (d-1)^{\omega(\D)}}\right)\\
& = O\left( B^n\sum_{\substack{u_1 | C(\xi_1) \\ u_2 \leq \xi_3}} \frac{  \sigma_{-1/4}(u_2)  \sigma_{-1/4}(\D)}{u_2 (d-1)^{\omega(D)}}\right).
\end{align*}
Observe that 
\begin{align*} \sigma_{-1/4}(\D) & = \prod_{p | \D} (1 + p^{-1/4}) \\
& = O\left(\left(\frac{2(d-1)}{3}\right)^{\omega(\D)} \right).
\end{align*}
It follows that
\[N_4(B) = O \left( B^n  \sum_{\substack{u_1 | C(\xi_1) \\ u_2 \leq \xi_3}} \frac{\sigma_{-1/4}(u_2)}{u_2 (3/2)^{\omega(D)}}\right).\]
Let us write 
\begin{equation}\label{xi4}\xi_4 = \xi_4(u_2) = \frac{1}{12} \log(B u_2^{-1}),
\end{equation}
and
\[\D' = \D'(u_2) = \prod_{\substack{p \leq \xi_4\\p\equiv 1(k)}} p.\]
%
%
For a square-free number $\ell$ we have
\[\sigma_0(\ell) = \prod_{p | l} (1 + 1) = 2^{\omega(\ell)}.\]
Then from the definitions of $\D$ and $\D'$ it follows that
\[(3/2)^{\omega(\D')} < (3/2)^{\omega(D)} C(\xi_1) (3/2)^{\gcd(\D', u_2)} < (3/2)^{\omega(\D)} C(\xi_1) \sigma_0(\gcd(\D', u_2).\]
Hence, there exists a positive number $C_2$ such that
\[\frac{1}{(3/2)^{\omega(\D)}} < \frac{C_2}{(3/2)^{\omega(\D')}} \sigma_0(\gcd(\D', u_2)) \exp(2(d-2) \xi_1).\]

From here we obtain the estimate
\begin{equation} \label{N4B pen} N_4(B) = O\left(\exp((2d-3)\xi_1) B^n\sum_{\substack{ u_2 \leq \xi_3}} \frac{  \sigma_{-1/4}(u_2)\sigma_0(\gcd(\D', u_2)) }{(3/2)^{\omega(\D')} u_2} \right),\end{equation}
where we used that the number of the divisors $u_1 | C(\xi_1)$ is $O(\exp(\xi_1))$.
We now estimate the sum
\[S(t) = \sum_{u_2 \leq t}  \sigma_{-1/4}(u_2) \sigma_0(\gcd(\D', u_2)).\]
We proceed as follows.
\begin{align}\label{S t} S(t) & \leq \sum_{h | \D'} \mu^2(h) \sigma_0(h) \sum_{\substack{u_2 \leq t \\ u_2 \equiv 0 \pmod{h}}} \sigma_{-1/4}(u_2) \\
& = \sum_{h | \D'} \mu^2(h) \sigma_0(h) \sum_{\substack{u_2' h \leq t \\ \gcd(u_2', h) = 1}}  \sigma_{-1/4} (h u_2') \notag \\
& \leq \sum_{h | \D'} \mu^2(h) \sigma_0(h)  \sigma_{-1/4} (h) \sum_{u_2' \leq t/h}  \sigma_{-1/4}(u_2') \notag \\
& = O \left(t \sum_{h | \D'} \frac{\mu^2(h) \sigma_0(h)  \sigma_{-1/4}(h)}{h} \right) \notag \\
& = O \left(t \prod_{\substack{p \leq \xi_4\\p\equiv 1(k)}} \left(1 + \frac{4}{p}\right) \right) \notag \\
& = O \left(t (\log \xi_4)^{4}\right) \notag \\
& = O \left(t (\log \log B)^{4}\right). \notag
\end{align}
Next, we have
\[\omega(\D') = \pi(\xi_4; k, 1) \sim \frac{\xi_4}{\varphi(k) \log \xi_4},\]
where $\pi(B; q, a)$ is the counting function of primes $p$ satisfying $p \equiv a \pmod{q}$ up to $B$, and the above asymptotic follows from 
the prime number theorem for arithmetic progressions. Therefore we may find a constant $C_3$ such that
\[\omega(\D') > \frac{C_3 \xi_4}{\log \xi_4}\] 
for all $B$ sufficiently large.\\

Next, we cut the range of the summation in (\ref{N4B pen}) into sub-ranges of the type $(e^{-\alpha-1}\xi_3,e^{-\alpha}\xi_3]$ for non-negative integer values of $\alpha$. In each of these sub-ranges we have by \eqref{xi4} that
\[\frac 1{12}\log\left(\frac{Be^{\alpha}}{\xi_3}\right)\leq \xi_4\leq \frac 1{12}\log\left(\frac{Be^{\alpha+1}}{\xi_3}\right).\]
We proceed by following Hooley's treatment of the term $N^{(6)}(X)$ from \cite[\S8]{Hoo1}. In each sub-range we bound the factor $(3/2)^{\omega(\D')}$ from \eqref{N4B pen} and also using \eqref{S t}, we can easily see that, just like in \cite[\S8]{Hoo1}, there exist some positive number $C_4$ such that
\begin{equation}\label{eq:N4_xi1}N_4(B) = O\left(\frac{B^n \exp((2d-3) \xi_1) }{(\log B)^{C_4/\log_3 B}}\right).\end{equation}
Recalling that after \eqref{xi1} we have $\exp((2k-1)\xi_1)\ll (\log_2 B) ^{C_5/\log_4 B}$ for some constant $C_5$, we get $N_4(B)=o(B^n)$ and this suffices for the proof of Theorem \ref{MT}.

\subsection*{Acknowledgements} 
K. Lapkova is supported by a Hertha Firnberg grant [T846-N35] of the  Austrian Science Fund (FWF).


\end{document}